\documentclass[12pt,a4paper]{amsart}
\usepackage{amsmath,amssymb,amsthm}
\usepackage{enumitem}
\usepackage{ifthen,verbatim}
\usepackage{mathrsfs}
\usepackage{color}
\usepackage{url}
\usepackage[english]{babel}
\usepackage{here}
\usepackage[T1]{fontenc}
\usepackage{floatflt,graphicx,graphics}
\usepackage{a4wide}
\usepackage{cite}
\usepackage{ifthen}

\newtheorem{thr}{Theorem}[section]
\newtheorem{lem}{Lemma}[section]
\newtheorem{rem}{Remark}[section]
\newtheorem{conjec}{Conjecture}[section]

\nonstopmode
\begin{document}
\title[Asymptotics of the exterior conformal modulus]{Asymptotics of the exterior conformal modulus\\ of a symmetric quadrilateral under stretching map}

\author{A.~Dyutin}
\address{Kazan Federal University, Kremlyovskaya str. 35, Tatarstan, 420008,
Russian Federation}
\email{dyutin.andrei2016@yandex.ru}

\author{Giang~V.~Nguyen}
\address{Kazan Federal University, Kremlyovskaya str. 35, Tatarstan, 420008,
Russian Federation}
\email{nvgiang.math@gmail.com}

\begin{abstract}
In this work, we study the distortion of the exterior conformal modulus of a symmetric quadrilateral, when stretched in the direction of the abscissa axis with the coefficient $H\to \infty$. By using some facts from the theory of elliptic integrals, we confirm that the asymptotic behavior of this modulus does not depend on the shape of the boundary of the quadrilateral; moreover, it is equivalent to $(1/\pi)\log H$ as $H\to \infty$.
\end{abstract}


\maketitle

\section{Introduction}
In the geometric function theory and its applications, an important problem is the study of conformal moduli of quadrilaterals and their distortions under quasiconformal mappings (see, e.g.,
\cite[Chap.~3, Sect.~11]{kuhnau2005conformal}, \cite{papamichael2010numerical}).

A quadrilateral $(Q; z_1, z_2, z_3, z_4)$ is a Jordan domain $Q$ on the Riemann sphere, on the boundary of which four different points (vertices) $z_1$, $z_2$, $z_3$, and $z_4$ are fixed;
they are located in such a way that the index increases with the positive bypass of the boundary.

The conformal modulus $\operatorname{Mod}(\mbox{\boldmath$Q$})$ of a given quadrilateral $\mbox{\boldmath$Q$}= (Q; z_1, z_2, z_3, z_4)$ is the extremal length $\lambda(\Gamma)$ of the family of
curves in the domain $Q$, connecting its boundary arcs $ \gamma_1 = (z_1, z_2)$ and $\gamma_2 = (z_3, z_4)$ (for extremal lengths, see, e.g., \cite[Chap.~1, Sect.~D]{ahlfors2006lectures}). Such the
modulus $\operatorname{Mod}(\mbox{\boldmath$Q$})$ is also known as the interior conformal modulus of $\mbox{\boldmath$Q$}$.

The conformal modulus of a quadrilateral is invariant under conformal transformations. It is also quasiinvariant under quasiconformal mappings (see, e.g., \cite[Chap.~2,
Sect.~A]{ahlfors2006lectures}, \cite[Chap.~3, Sect.~22]{kuhnau2005conformal}): If $f_{\mathcal{K}}$ is an $\mathcal{K}$-quasiconformal mapping, then
\begin{align*}
\dfrac{1}{\mathcal{K}}\operatorname{Mod}(\mbox{\boldmath$Q$})\leq \operatorname{Mod}(f_{\mathcal{K}}(\mbox{\boldmath$Q$}))\leq \mathcal{K} \operatorname{Mod}(\mbox{\boldmath$Q$})
\end{align*}
where $\mbox{\boldmath$Q$}$ is some given quadrilateral.

By the Riemann mapping theorem, there exists a conformal mapping of the domain $Q$ onto the rectangle $\Pi = \{w\,:\, 0 <\operatorname{Re} w < 1, 0 < \operatorname{Im} w < m \}$ translating the
arcs $\gamma_1$ and $\gamma_2$ to the horizontal sides of $\Pi$. Then, the conformal modulus $\operatorname{Mod}(\mbox{\boldmath$Q$})$ of the quadrilateral $\mbox{\boldmath$Q$}$ is equal to $m$.

Finally, we can determine the conformal modulus by using the Dirichlet integral (see, e.g. \cite{dubinin2014condenser}). For a given quadrilateral $\mbox{\boldmath$Q$}=(Q; z_1, z_2, z_3, z_4)$, let
$\mathcal{L}$ be the class of real-valued functions $u$ which are continuous on $\overline{Q}$ and belong to the Sobolev space $W_1(Q) $ with boundary values $u = 0$ and $u = 1$ on $(z_1, z_2)$ and $(z_3, z_4)$, respectively. Then
$$
\left({\operatorname{Mod}(\mbox{\boldmath$Q$})}\right)^{-1}=\min\limits_{u\in\mathcal{L}}D_{Q}[u],
$$
where
$$
D_{Q}[u]=\int\limits_Q|\operatorname{grad} u|dxdy=\int\limits_Q\left[\left(\frac{\partial u}{\partial x}\right)^2+\left(\frac{\partial u}{\partial y}\right)^2\right]dxdy
$$
is the Dirichlet integral of the function $u$ over the domain $Q$.

\begin{rem}
The conformal modulus of a quadrilateral $\mbox{\boldmath$Q$}= (Q; z_1, z_2, z_3, z_4)$ can also be defined when $Q$ is not a Jordan domain. For example, if $Q$ is a simply connected domain with a
nondegenerate boundary, and $z_1$, $z_2$, $z_3$, $z_4$ are four different boundary prime ends of $Q$ (see, e.g., \cite[Chap.~2, Sect.~3]{goluzin1969geometric}), then the conformal modulus can be
defined in the same way.
\end{rem}

We need to note the well-known monotonicity property of the conformal modulus (see, e.g., \cite[Chap.~2, Sect.~3]{papamichael2010numerical}).

\begin{lem}\label{monot}
If we change the domain $Q$ by keeping the boundary arcs, $(z_1,z_2)$ and $(z_3,z_4)$, and the vertices of the quadrilateral unchanged, while pushing $(z_2,z_3)$ and $(z_4,z_1)$ into the domain $Q$,
then the
conformal modulus of the quadrilateral increases. Under the reverse transformation, the conformal modulus, naturally, decreases.
\end{lem}

If the domain $Q$ does not contain a point at infinity in its closure, then along with the coformal modulus of the quadrilateral $\mbox{\boldmath$Q$}= (Q; z_1, z_2, z_3, z_4)$, we can consider the so-called exterior conformal modulus of this quadrilateral. Denote by $Q^c=\overline{\mathbb{C}}\setminus Q$ the complement of $Q$ in the extended complex plane. Then $\infty\in Q^c$ and the conformal modulus of the quadrilateral $\mbox{\boldmath$Q$}^c=(Q^c; z_4, z_3, z_2, z_1)$ is called the exterior conformal modulus of $\mbox{\boldmath$Q$}$. We will denote this conformal modulus by $\operatorname{ExtMod}(\mbox{\boldmath$Q$})$.

In 1993, P.~Duren and J.~Pfaltzgraff \cite{duren1993robin} actually found a formula connecting the interior and exterior conformal moduli for the the case of the rectangle $\Pi_{ab}:=[-a/2,a/2]\times[0, b]$, $a$, $b\in (0, \infty)$, whose vertices coincide with the natural ones of $\Pi_{ab}$, i.e. $z_1=-a/2+ib$, $z_2=-a/2$, $z_3=a/2$, $z_4=a/2+ib$. They proved that if $k$, $0<k<1$, is such that
\begin{equation}\label{(2)}
\operatorname{ExtMod}(\Pi_{ab}) = \frac{2 K(k)}{K'(k)},
\end{equation}
then the interior conformal modulus of $\Pi_{ab}$, i.e. the aspect ratio $a/b$ is equal to
\begin{equation}\label{(1)}
\operatorname{Mod}(\Pi_{ab})=\Psi(k):=\frac{2\left[E(k) - (1 - k) K(k)\right]}{E'(k) - k K'(k)},
 \end{equation}
where $K(k)$ and $E(k)$ are the complete elliptic integrals of the first and second kinds; $K'(k)=K(k')$, $E'(k)=E(k ')$, $k':=\sqrt{1-k^2}$ (see Section~\ref{ell} below for details).

We note that earlier, in 1932, W.~Bickley \cite{bickley} also obtained similar formulas from which we can deduce the indicated result of P.~Duren and J.~Pfaltzgraff (see Section~\ref{sec_bickley}
below for details).

Using the results of \cite{duren1993robin}, M.~Vuorinen and H.~Zhang \cite[Theorem 4.3]{vuorinen2013exterior} established upper and lower bounds for the exterior conformal modulus of the rectangle
$\Pi_{ab}$. In particular, they proved the following inequality
\begin{equation}\label{extm}
\frac{2}{\pi}\,\left(1-(1+\sqrt{4H/\pi})^{-1}\right)\log\left(2(1+\sqrt{4H/\pi})\right)<\operatorname{ExtMod}(\Pi_{ab})<\frac{2}{\pi}\,\log\left(2(1+\sqrt{\pi H})\right),
\end{equation}
where $H=a/b=\operatorname{Mod}(\Pi_{ab})$. As remarked in \cite{vuorinen2013exterior}, the exterior conformal modulus $\operatorname{ExtMod}(\Pi_{ab})$ has the logarithmic growth with respect to $H$.
Furthermore, the following result is deduced from~\eqref{extm}.

\begin{thr}\label{extmodr} Let  $\Pi_{ab}:=[-a/2,a/2]\times[0,b]$ and $H=a/b$. Then
\begin{equation}\label{extab}
 \operatorname{ExtMod}(\Pi_{ab})\sim \frac{1}{\pi}\,\log H, \quad\text{ as }\quad H\to\infty.
\end{equation}
\end{thr}

The exterior conformal moduli for quadrilaterals of a sufficiently arbitrary shape and their properties have been studied in \cite{nguyen2022}, in connection with solving the Vuorinen problem on the
asymptotic behavior of the conformal modulus of an unbounded doubly connected domain under its infinite stretch. Also in the article by S.~Nasyrov, T.~Sugawa and M.~Vuorinen
\cite{nasyrov2021moduli} the exterior conformal moduli of quadrilaterals, in particular isosceles trapezoids, have been investigated.

Let us formulate the Vuorinen problem on the asymptotic of conformal moduli of doubly connected domains. First, we recall one of the possible definitions of the conformal modulus of a doubly
connected domain $\Omega$ with non-degenerate components. If $\Omega$ is conformally equivalent to an annulus $\{z\in\mathbb{C}\,:\, 1<|z|<q\}$ (see, e.g., \cite[Chap.~V, Sect.~1]{goluzin1969geometric}, \cite[Sect.~3]{kuhnau2005conformal}), then its conformal modulus is equal to $(1/2\pi)\log q$. Vuorinen's problem is posed as follows.

Let $\Omega$ be a doubly connected domain in the plane and $f_H: x+iy\mapsto Hx+iy$ be an $H$-quasiconformal mapping of the complex plane with the coefficient $H>0$. Denote by $\Omega_H$ the
image of the domain $\Omega$ under this map. M.~Vuorinen asked: \textquotedblleft {\it What is the behavior of the conformal modulus of $\Omega_H$ as $H\to \infty$\,?} \textquotedblright. Regarding
this problem, it becomes necessary to describe the asymptotic behavior of the interior and exterior conformal moduli of quadrilaterals, which are obtained from a given quadrilateral under
stretching map $f_H$ as $H\to \infty$.

Vuorinen's problem have been studied for the cases of rectangular frames \cite{nasyrov2015riemann}, for bounded doubly connected domains \cite{dautova2018, dautova2019}, and also for unbounded ones
\cite{nvgiang2021, nguyen2022}. Especially, in \cite{dautova2019, nvgiang2021, nguyen2022} a formula has been obtained describing the asymptotic behavior of the interior conformal modulus of a quadrilateral. Let the functions $f$ and $g$ be continuous on the segment $[a,b]$, $-\infty< a<b<+\infty$, and $f(x)<g(x)$ for all $x\in[a,b]$. Consider a quadrilateral $\mbox{\boldmath$Q$}= (Q; z_1, z_2, z_3, z_4)$, where $Q$ is bounded by two vertical segments $[z_1,z_2]$ and $[z_3,z_4]$ ending at the points $z_1=a+ig(a) $, $z_2=a+if(a)$, $z_3=b+if(b)$, $z_4=b+ig(b)$, and two curves
$$
\Gamma_1 = \{x + iy\,:\,y = f(x),\;a \leq x \leq b\},\quad \Gamma_2 = \{x + iy\,:\,y = g(x),\;a \leq x \leq b\}.
$$
The class of such quadrilaterals will be denoted by $\mathfrak{S}$. Denote by $\mbox{\boldmath$Q$}_H$ the quadrilateral which is the image of $\mbox{\boldmath$Q$}$ under the mapping~$f_H$. The
following result is proved in \cite[Theorem.~4.2]{nvgiang2021}.
\begin{thr}\label{inner}
Let the quadrilateral $\mbox{\boldmath$Q$}\in \mathfrak{S}$. Then the following asymptotic formula holds:
$$
 \operatorname{Mod}(\mbox{\boldmath$Q$}_H)\sim \frac{1}{c H},\ \ H\to\infty, \quad \text{where} \quad c=\int\limits_a^b\frac{dx}{g(x)-f(x)}.
$$
\end{thr}

The similar problem for the exterior conformal modulus of quadrilaterals is still open. Based on Theorem~\ref{extmodr} and analyzing some examples, given in \cite{nasyrov2021moduli},  S.~R.~Nasyrov proposed the following conjecture.

\begin{conjec}[S.~R.~Nasyrov]\label{conject_nasyrov}
When an arbitrary quadrilateral $\mbox{\boldmath$Q$}\in\mathfrak{S}$ is stretched along the abscissa axis, the asymptotic behavior of the exterior conformal modulus is similar to \eqref{extab}, i.e. does not depend on the shape of the boundary of $Q$.
\end{conjec}

In this paper, we confirm the validity of Conjecture~\ref{conject_nasyrov} for the case of the quadrilateral $\mbox{\boldmath$Q$}=(Q; z_1, z_2, z_3, z_4)\in\mathfrak{S}$ such that $Q$ is symmetric with respect to both the coordinate axes. We establish

\begin{thr}\label{mt}
Let a quadrilateral $\mbox{\boldmath$Q$}\in\mathfrak{S}$ be symmetric with respect to both the coordinate axes. Then
\begin{equation}\label{asqh}
 \operatorname{ExtMod}(\mbox{\boldmath$Q$}_H)\sim \frac{1}{\pi}\,\log H, \quad\text{as }\quad H\to\infty.
\end{equation}
\end{thr}

Now we will describe the structure of the paper. In Section~\ref{ell} we give some basic facts about the elliptic integrals. In Section~\ref{sec_bickley} we study in detail the case of rectangular
domains. We give another proof of Theorem~\ref{extmodr} based on the properties of the elliptic integrals, as well as compare the formulas relating the interior and exterior conformal moduli of
rectangles from \cite{bickley} and \cite{duren1993robin}, and show that these formulas are equivalent. At last, in Sections~\ref{q1} we obtain the upper and lower bounds of the exterior conformal
modulus of $\mbox{\boldmath$Q$}_H$ in the case of domain $\mbox{\boldmath$Q$}\in\mathfrak{S}$ such that $Q$ is symmetric with respect to both the coordinate axes. For this, we find two domains,
$G_{1H}$
and $G_{2H}$, such that  $G_{1H}\subset Q_H^c\subset G_{2H}$, and show that $G_{1H}$ and $G_{2H}$ satisfy the asymptotic formula \eqref{asqh}. This immediately implies the statement of
Theorem~\ref{mt}.

\section{Elliptic integrals}\label{ell}
Let us recall some facts from the theory of elliptic integrals.

In the Legendre normal form, the incomplete elliptic integrals of the first and second kinds are defined by the formulas
\begin{equation}\label{incomplete}
F(z, k) = \int\limits_0^z \frac{dt}{\sqrt{(1 - t^2)(1 - k^2 t^2)}},\quad
E(z, k) = \int\limits_0^z \sqrt{\frac{1 - k^2 t^2}{1 - t^2}}\,dt,
\end{equation}
where $k \in (0, 1)$ is a parameter.

Putting $z=1$ in \eqref{incomplete},  we obtain the following integrals:
\begin{align*}
K(k)= \int\limits_0^1 \frac{dt}{\sqrt{(1 - t^2)(1 - k^2 t^2)}}, \quad E(k)= \int\limits_0^1 \sqrt{\frac{1 - k^2 t^2}{1 - t^2}}\,dt.
\end{align*}
They are called the complete elliptic integrals of the first and the second kinds. We denote $K'(k) = K(k')$ and $E'(k) = E(k')$ where $k':= \sqrt{1 - k^2}$.

For $k \in(0, 1)$, the following expansions into the power series in the variable~$k$ take place (see, e.g., \cite{byrd_fridman}):
\begin{equation}\label{exp_K}
K(k) = \frac{\pi}{2}\left\{1 + \sum_{n=1}^\infty\left[\frac{(2n - 1)!!}{(2n)!!}\right]^2 {k^{2n}} \right\},
\end{equation}
\begin{equation}\label{exp_E}
E(k) = \frac{\pi}{2}\left\{1 -  \sum_{n=1}^\infty\frac{1}{2n - 1}\left[\frac{(2n - 1)!!}{(2n)!!}\right]^2 {k^{2n}}\right\}.
\end{equation}
For $k$ close to $1$, the following expansions are valid:
\begin{align}
K'(k)&= \ln \frac{4}{k'} +\frac{1}{4} \left(\ln \frac{4}{k'} - 1\right){k'}^2 + \frac{9}{64}\left(\ln \frac{4}{k'} - \frac{4}{3}\right){k'}^4 +
\ldots,\label{exp_K'}\\
E'(k)&= 1 + \frac{1}{2}\left(\ln \frac{4}{k'} - \frac{1}{2}\right){k'}^2 + \frac{3}{16}\left(\ln \frac{4}{k'} - \frac{13}{12}\right){k'}^4 + \ldots\,.\label{exp_E'}
\end{align}
From \eqref{exp_K}, \eqref{exp_E}, \eqref{exp_K'} and  \eqref{exp_E'} it follows that
\begin{equation*} \lim_{k \to 0} K(k) = \frac{\pi}{2}\,,\;\;\;\lim_{k \to 0} E(k) = \frac{\pi}{2}\,,\;\;\;\lim_{k \to 0} \left(K'(k) - \text{ln}\,\frac{4}{k} \right) = 0,\;\;\;\lim_{k \to 0} E'(k)= 1,
\end{equation*}
\begin{equation}\label{asympk}
\lim_{k \to 1} \left(K(k) - \text{ln}\,\frac{4}{k'} \right) = 0,\;\;\;\lim_{k \to 1} E(k) = 1,\;\;\;\lim_{k \to 1} K' (k) = \frac{\pi}{2}\,,\;\;\;\lim_{k \to 1} E'(k) = \frac{\pi}{2}.\end{equation}

\section{Duren--Pfaltzgraff and Bickley formulas for a rectangle}\label{sec_bickley}

As we noted in Introduction, P.~Duren and J.~Pfaltzgraff \cite{duren1993robin} stated that if the exterior conformal modulus of the rectangle $\Pi_{ab}=[-a/2,a/2]\times[0,b]$ is $2K(k)/K(k ')$, then
$\operatorname{Mod}(\Pi_{ab})=\Psi(k)$, where the function $\Psi$ has the form \eqref{(1)}. They showed that the function $\Psi:(0,1) \to (0,\infty)$ is a homeomorphism, in particular,
$\Psi^{-1}:\;(0,
\infty) \to (0, 1)$ is well-defined, and study the behavior of $\Psi$ at the point $k=1$. The following result is given in \cite{duren1993robin} without a detailed proof:
\begin{equation}\label{(3)}
\Psi(k) \sim \frac{16}{\pi (1 - k)^2}\,,\quad\text{ as }\quad k\to 1.
\end{equation}

We give a brief proof of this fact. Due to \eqref{exp_K} and \eqref{exp_E} we have
\begin{equation}\label{numer}
\lim_{k\to1}\left[E(k) - (1 - k)K(k)\right]=1.
\end{equation}
We also note that
$$
E(k) = \frac{\pi}{2}\left(1 - \frac{3}{64}\, k^2 - \frac{7}{64}\, k^4 +\dots \right),\quad k'K(k) = \frac{\pi}{2}\left(1 - \frac{1}{4}\, k^2 - \frac{7}{64}\, k^4 +\dots \right),
$$
therefore, $E(k)-k'K(k)\sim \frac{\pi}{32}\, k^4$ as $k \rightarrow 0$. Replacing $k$ with $k'$, we get
\begin{equation}\label{denom}
E'(k) - kK'(k) \thicksim \frac{\pi}{32} (1- k^2)^2 \thicksim \frac{\pi}{8} (1- k)^2, \quad k\to 1.
\end{equation}
From \eqref{(1)}, \eqref{numer} and \eqref{denom} we deduce~\eqref{(3)}.

Note that from \eqref{(3)} we can also obtain Theorem~\ref{extmodr}. Indeed, as $k\to 1$, we have
$$
\text{ExtMod}(\Pi_{ab})=\frac{2K(k)}{K'(k)}\sim \frac{4}{\pi}\,K(k)\sim\frac{4}{\pi}\log\frac{4}{k'}\sim\frac{2}{\pi}\log\frac{1}{1-k}\sim\frac{1}{\pi}\log\Psi(k)=\frac{1}{\pi}\log H.
$$

Now we will formulate the result obtained by W.~Bickley \cite{bickley} on the relationship between the interior and exterior conformal moduli of a rectangle $\Pi_{ab}$. He proved that if
$$
\text{ExtMod}(\Pi_{ab})=\frac{K(\lambda)}{K'(\lambda)}
$$
for some parameter $\lambda\in(0,1)$, then the interior conformal modulus is expressed by the formula
\begin{equation}\label{(11)}
\operatorname{Mod}(\Pi_{ab}) = \frac{E(\lambda) - {{\lambda}'}^2 K(\lambda)}{E'(\lambda) - {\lambda}^2 K'(\lambda)}.
\end{equation}

Now, we will show that the Duren-Pfaltzgraff formula follows from W.~Bickley's result. Actually, if we put
\begin{equation}\label{(15)}
\lambda = \frac{2 \sqrt{k}}{1 + k},\;\;\;\lambda' =\sqrt{1-\lambda^2}=\frac{1 - k}{1 + k},
\end{equation}
then using the Gauss--Landen transformation we obtain (see, e.g., \cite[Formula~(3.15), p.~51]{anderson1997conormal})
\begin{align*}
K(\lambda)=(1+k)K(k), &\qquad K'(\lambda)=\frac{1+k}{2}\,K'(k),\\
E(\lambda)=\frac{2E(k)-k'^2K(k)}{1+k}, &\qquad E'(\lambda)=\frac{E'(k)+kK'(k)}{1+k}.
\end{align*}
From these equalities we deduce that
$$
\frac{K(\lambda)}{K'(\lambda)}=\frac{2K(k)}{K'(k)},
$$
$$
\frac{E(\lambda) - {{\lambda}'}^2 K(\lambda)}{E'(\lambda) - {\lambda}^2 K'(\lambda)}=\frac{2\left[E(k) - (1 - k) K(k)\right]}{E' (k) - k K' (k)}.
$$
Thus, \eqref{(1)} directly follows  from \eqref{(11)}. We need to emphasize that in both the papers, \cite{bickley} and \cite{duren1993robin}, the same Schwarz-Christoffel integral was used, but
further W.~Bickley applied  additionally the Landen transformation \cite[p.~85]{bickley}, which is expressed by \eqref{(15)} relating the parameters $\lambda$ and $k$ of the elliptic integrals.

%%%%%%%%%%%%%%%%%%%%%%%%%%%%%%%%%%%%%%%%%%%%%%%%%%%%%%%%%%%%%%%%%%%%%%%%%%%%%%%%%%%%%%%%%%%
\section{Asymptotics of the exterior conformal modulus}\label{q1}
Consider the quadrilateral $\mbox{\boldmath$Q$} = (Q; A, B, C, D)$ of class $\mathfrak{S}$, such that $Q$ is symmetric with respect to both the coordinate axes. Then, the domain $Q$ is bounded by
the curves $y=f(x)$, $y=-f(x)$, $|x|\le\alpha$, where $f$ is a continuous positive even function on the segment $[- \alpha,\alpha]$, and vertical segments $AB$ and $CD$, $A=-\alpha+i\beta$,
$B=-\alpha-i\beta$, $C=\alpha- i\beta$, $D=-\alpha+i\beta$, and $\beta=f(\alpha)$. Denote by $\mbox{\boldmath$Q$}_H = (Q_H; A_H, B_H, C_H, D_H)$ its image under the stretching map $f_H$. By
definition, the exterior conformal modulus of the quadrilateral $\mbox{\boldmath$Q$}_H$ is equal to the conformal modulus of the quadrilateral $\mbox{\boldmath$G$}_H=(G_H; D_H, C_H, B_H, A_H)$,
where
$G_H=Q_H ^c$ is the complement of $Q_H$ in the extended complex plane. Now we will estimate the conformal modulus of $\mbox{\boldmath$G$}_H$ in terms of the moduli of other quadrilaterals.

Fixed a number $M>\max\limits_{x\in[-\alpha,\alpha]}f(x)$, consider the quadrilateral
$\mbox{\boldmath$G$}_{1H}=(G_{1H}; D_H,C_H,B_H,A_H)$, where $G_{1H}$ is the complement of the rectangle $[-H\alpha, H\alpha]\times [ -M, M]$ in the extended complex plane. Note that the vertices of
the quadrilateral $\mbox{\boldmath$G$}_{1H}$ coincide with the vertices $D_H$, $C_H$, $B_H$, $A_H$ of the quadrilateral $\mbox{\boldmath$G$}_{H}$. Since $G_{1H}\subset G_{H}$, by
Lemma~\ref{monot}, we have
\begin{align}\label{eq_upper}
\operatorname{Mod}(\mbox{\boldmath$G$}_{H})\le\operatorname{Mod}(\mbox{\boldmath$G$}_{1H}).
\end{align}

Now we consider the quadrilateral $\mbox{\boldmath$G$}_{2H}=(G_{2H}; D_H, C_H, B_H, A_H)$ whose domain is the complement of the union of three the segments, $A_H B_H$, $C_H D_H$ and $[- H\alpha,
H\alpha]$. Then $G_{H}\subset G_{2H}$ and, by Lemma \ref{monot}, we obtain
\begin{align}\label{eq_lower}
\operatorname{Mod}(\mbox{\boldmath$G$}_{2H})\le\operatorname{Mod} (\mbox{\boldmath$G$}_{H}).
\end{align}
Since $\operatorname{Mod}(\mbox{\boldmath$G$}_{H})=\operatorname{ExtMod}(\mbox{\boldmath$Q$}_{H})$, from \eqref{eq_upper} and \eqref{eq_lower} we deduce that
\begin{equation}\label{compar}
\operatorname{Mod}(\mbox{\boldmath$G$}_{2H})\le\operatorname{ExtMod}(\mbox{\boldmath$Q$}_{H})\le\operatorname{Mod}(\mbox{\boldmath$G$}_{1H}).
\end{equation}

Now we will study the asymptotic of the conformal moduli of the quadrilaterals $\mbox{\boldmath$G$}_{1H}$ and $\mbox{\boldmath$G$}_{2H}$.

\begin{lem}\label{lem_asym_G1H}
We have
\begin{align}\label{mod_G1H}
\operatorname{Mod}(\mbox{\boldmath$G$}_{1H})\sim\frac{1}{\pi}\,\log H,\quad\text{ as }\quad H\to\infty.
\end{align}
\end{lem}
\begin{proof} Consider the quadrilateral $\mbox{\boldmath$G$}^{*}_{1H}$ whose domain coincides with $G_{1H}$ and the vertices of which coincide with those of the rectangle $[-H\alpha,
H\alpha]\times[-M, M]$. Let the function $\varphi_H$ 
map conformally the upper half-plane $\mathbb{H}_{\zeta}^+$ of the $\zeta$-plane onto the domain $G_{1H}$ such that for some $k=k(H)\in (0, 1)$
we have the following correspondence of the boundary points:
\begin{equation*}
\varphi_H(\pm 1/\sqrt{k})=\pm H\alpha-iM, \quad \varphi_H(\pm\sqrt{k})=\pm H\alpha+iM.
\end{equation*}
Using the Riemann-Schwarz symmetry principle, we can verify that for some $\lambda=\lambda(H)\in (k,1)$ the points $-1/\sqrt{\lambda}$, $-\sqrt{\lambda }$, $\sqrt{\lambda}$, $1/\sqrt{\lambda}$
correspond to the vertices of the quadrilateral
$\mbox{\boldmath$G$}_{1H}$.

The conformal modulus is invariant under conformal mappings, therefore, $\operatorname{Mod}(\mbox{\boldmath$G$}_{1H})$ coincides with the conformal modulus of the quadrilateral, which is the
upper half-plane with marked points (vertices) $-1/\sqrt{\lambda}$, $-\sqrt{\lambda}$, $\sqrt{\lambda}$, $1/\sqrt{\lambda}$.

The function $\zeta\mapsto \zeta /\sqrt{\lambda}$  is a conformal automorphism of $\mathbb{H}_{\zeta}^+$, mapping the points $\pm 1/\sqrt{\lambda}$, $\pm \sqrt{\lambda}$ to the points
$\pm1/\lambda$, $\pm 1$. Since the elliptic integral
$$
F(\zeta, \lambda) =\int\limits_0^{\zeta}\frac{dt}{\sqrt{(1-t^2)(1-\lambda^2t^2)}}
$$
maps conformally the upper half-plane $\mathbb{H}_{\zeta}^+$ onto the rectangle $[-K(\lambda), K(\lambda)]\times[0,K'(\lambda)]$ and the points $\pm1/\lambda$, $\pm 1$ correspond to the vertices of
the rectangle, we get
$$
\operatorname{Mod}(\mbox{\boldmath$G$}_{1H})= \frac{2K(\lambda)}{K'(\lambda)}.
$$
Similarly, we verify that
$$
\operatorname{Mod}(\mbox{\boldmath$G$}^{*}_{1H})=\frac{2K(k)}{K'(k)}.
$$

Using Theorem~\ref{extmodr}, we get
\begin{align}\label{mod_G*}
\operatorname{Mod}(\mbox{\boldmath$G$}^{*}_{1H})\sim \frac{1}{\pi}\,\log \frac{\alpha H}{M}\sim\frac{1}{\pi}\,\log H, \quad\text{ as }\quad H\to\infty.
\end{align}

Now, we will show that $\operatorname{Mod}(\mbox{\boldmath$G$}_{1H})\sim\operatorname{Mod}(\mbox{\boldmath$G$}^{*} _{1H})$, i.e.
\begin{equation}\label{kl}
\frac{2K(k)}{K'(k)}\sim\frac{2K(\lambda)}{K'(\lambda)}
\end{equation}
for $H\to\infty$.

Obviously, the values $k=k(H)$ and $\lambda=\lambda(H)$ tend to $1$ as $H\to\infty$. Moreover
$$
\frac{2K(k)}{K'(k)}\sim \frac{2}{\pi}\log\frac{1}{1- k}, \quad k\to 1, \quad\text{and} \quad \frac{2K(\lambda)}{K'(\lambda)}\sim \frac{2}{\pi}\log\frac{1}{1- \lambda}, \quad \lambda\to 1.
$$
Thus, to prove \eqref{kl} it is sufficient to establish that
\begin{equation}\label{kl1}
\log\frac{1}{1-k(H)} \sim \log\frac{1}{1-\lambda(H)}, \quad\text{ as }\quad H\to\infty.
\end{equation}

To prove \eqref{kl1}, consider the family of functions
$$
\psi_{1H}(w)=\varphi_H\left(1+(1/\sqrt{k(H)}-1)w\right)-\alpha H.
$$
For every fixed $H$, the function $\psi_{1H}$ maps conformally the upper half-plane $\mathbb{H}_{w}^+$ onto the domain $G'_{1H}$ that is the exterior of the rectangle $[-2H\alpha,0]\times[- M,M]$, such that the points $0$, $1$ and $\infty$ go to the points $0$, $-iM$ and $-H\alpha-iM$. As $H\to \infty$, the family of domains $G'_{1H}$ converges to the half-strip $G'_{1}:=(-\infty,0]\times[-M, M]$
as a kernel. We also note that between the boundaries $\partial G'_{1}$ and $\partial G'_{1H}$ we can establish a homeomorphic correspondence $\eta_H:\partial G'_{1}\to \partial G'_{1H}$ such that $\eta_H\rightrightarrows \eta$ on $\partial G'_{1}$ in the spherical metric, where $\eta$ is the identity mapping of $\partial G'_1$. Moreover, $\psi_{1H}(0)=0$ and $\psi_{1H}(1)=-iM$ do not depend on $H$, while $\psi_{1H}(\infty)\to\infty$,  as $H\to\infty$. By the generalized Rad{\'o} theorem on the uniform convergence of conformal mappings (see, e.g., \cite[Theorem.~5]{dautova2018},
\cite[Theorem.~3]{dautova2019}, \cite[Theorem.~14.6, p.~191]{nasyrov2008}), the family $\psi_{1H}$ converges uniformly in the closure $\overline{\mathbb{H}}^+_w$  of the upper half-plane to the
conformal mapping $\psi_1$ of this half-plane onto $G'_1$ such that the points $0$, $1$, and $\infty$ correspond to the points $0$, $-iM$ and $\infty$.

Additionally, we see that $\psi_{1H}$ maps the point $\frac{1/\sqrt{\lambda(H)}-1}{1/\sqrt{k(H)}-1}$ to the point $-i\beta$, not dependent of $H$, therefore, as $H\to\infty$, we have
\begin{equation}\label{sqrt_lk}
\frac{1/\sqrt{\lambda(H)}-1}{1/\sqrt{k(H)}-1}\to \delta_1
\end{equation}
where $\delta_1 = \psi^{-1}_{1}(-i\beta)>0$. Then from \eqref{sqrt_lk}, taking into account that the quantities $\lambda(H)$ and $k(H)$ tend to unity, we deduce that $ \frac{1-\lambda(H)}{1-k( H)}
\to \delta_1$, as $H\to\infty$, which implies \eqref{kl1}.
\end{proof}

Now we study the asymptotics of the conformal modulus of the quadrilateral $\mbox{\boldmath$G$}_{2H}$.

\begin{lem}\label{lem_asym_G2+H}
We have
\begin{align}\label{mod_asym_G2+H}
\operatorname{Mod}(\mbox{\boldmath$G$}_{2H})\sim\frac{1}{\pi}\,\log H,\quad\text{ as }\quad H\to\infty.
\end{align}
\end{lem}

\begin{proof}
According to \cite[Formula 119.03, p.~18]{byrd_fridman}, the conformal mapping of the exterior of two horizontal segments, $[-1/k,-1]$ and $[1,1/k]$ ($0<k<1$) onto the exterior of two vertical
segments $\Delta_1=[-H\alpha -i\beta, -H\alpha +i\beta ]$ and $\Delta_2=[H\alpha -i\beta, H\alpha +i\beta]$ has the form
\begin{equation}\label{zetaw1}
\zeta(w) =C \int\limits_0^w \frac{\left(1/\lambda^2-t^2\right) dt}{\sqrt{(1 - t^2)(1 - k^2 t^2)}}\,.
\end{equation}
Here $1/\lambda \in(1, 1/k)$ is the preimage of the upper endpoint $H\alpha +i\beta $ of the segment $\Delta_2$. It is easy to see that
$$
\zeta(w) =\frac{C}{k^2} \left[E(w,k)-(1-k^2/\lambda^2)F(w,k) \right],
$$
where $E(w,k)$ and $F(w,k)$ are the incomplete elliptic integrals of the first and second kinds. Moreover, we have  \cite{byrd_fridman}:
\begin{equation}\label{Hab}
H\alpha =\frac{C}{k^2} \left[E(k)-(1-k^2/\lambda^2)K(k) \right],\quad \beta =\frac{C}{k^2} \left[E(l',k')-(k^2/\lambda^2)F(l',k') \right],
\end{equation}
where
\begin{equation}\label{k1}
l'=\sqrt{1-l^2}, \quad l=\frac{k}{k'}\,\sqrt{\frac{1}{\lambda^2} -1}, \quad \frac{1}{\lambda^2}=\frac{1}{k^2} \frac{E'(k)}{K'(k)}.
\end{equation}
We note that for every $H>0$ there are unique $k=k(H)$, $\lambda=\lambda(H)$, $l=l(H)$ and $C=C(H)$ such that
\eqref{zetaw1}, \eqref{Hab} and \eqref{k1} take place, if $\alpha$ and $\beta$ are fixed; moreover,
$k=k(H)\to 1$ and  $\lambda=\lambda(H)\to 1$ as $H\to\infty$. Since all the parameters in \eqref{zetaw1} are uniquely defined by $H$, we will write $\zeta(w)=\zeta_H(w)$.

By the Riemann-Schwarz symmetry principle, $\zeta_H(w)$  maps the upper half-plane $\mathbb{H}_{\zeta}^+$ onto $G_{2H}$; below we will only  consider $\zeta_H(w)$  on $\mathbb{H}_{\zeta}^+$.

From \eqref{Hab} it follows that
\begin{equation}\label{kkprime}
\frac{E(k)-(1-k^2/\lambda^2)K(k)}{E(l',k')-(k^2/\lambda^2) F(l',k')}=\frac{H\alpha}{\beta}.
\end{equation}

As in the proof of Lemma~\ref{lem_asym_G1H}, we show that
\begin{equation}\label{lambdak}
\frac{1-\lambda(H)}{1-k(H)}\to \delta_2\in (0,1), \quad\text{ as }\quad  H\to\infty.
\end{equation}
Actually, consider the family of functions $\psi_{2H}(z)=\zeta_H(1+(1/k-1)z)-\alpha H$. For every $H>0$ the function $\psi_{2H}$ maps the upper half-plane onto the domains $\mathcal{G}_H$ that is the half-plane with two the excluded segments, $[0,0+i\beta]$ and $[-2H\alpha,-2H\alpha +i\beta]$. In addition, $\psi_{2H}$ maps the point $0$, $1$, and $\infty$ to $0^-_H$, $0^+_H$, and $\infty$, where $0^-_H$ and $0^+_H$ are the prime ends of $\mathcal{G}_H$ with support at the origin, lying on the left and right edges of the slit along the segment $[0,0+i\beta]$. As $H\to\infty$, the family $\mathcal{G}_H$  converges to the kernel $\mathcal{G}$ that is the half-plane with the excluded segment $[0,0+i\beta]$. Then, by the generalized Rad\'o theorem, $\psi_{2H}$ converges uniformly to $\psi_2$ that is the conformal mapping of the upper half-plane onto $\mathcal{G}$ such that $0$, $1$, and $\infty$ are mapped to $0^-$, $0^+$, and $\infty$, where $0^-$ and $0^+$ are the prime ends of $\mathcal{G}$ with support at $0$, lying on the left and right edges of the slit along the segment $[0,0+i\beta]$. Since $\psi_{2H}((1/\lambda(H)-1)/(1/k(H)-1))=i\beta$, we conclude that $(1/\lambda(H)-1)/(1/k(H)-1)\to \delta_2$, where $\delta_2$ is the preimage of $i\beta$ under the mapping $\psi_2$. Therefore, $0<\delta_2<1$ and \eqref{lambdak} holds.

From \eqref{k1} and \eqref{lambdak} we conclude that $l'=l'(H)\to \sqrt{\delta_2}\in(0,1)$, as $H\to\infty$.

From \eqref{asympk}, taking into account the fact that $k<\lambda<1$, we deduce that  for $k=k(H)$ and $\lambda=\lambda(H)$ we have
\begin{align}\label{nenoH}
E(k)-(1-k^2/\lambda^2)K(k)\to 1, \quad\text{as} \quad H\to \infty.
\end{align}

On the other hand, from \eqref{incomplete} we have the series expansions (see, e.g., \cite[p.~300]{byrd_fridman}):
\begin{align}\label{exp_int1}
F(x, k)=\sum\limits_{n=0}^{\infty}\frac{(2n)!}{2^{2n}{n!}^2}\,k^{2n}I_{2n}(x),
\end{align}
\begin{align}\label{exp_int2}
E(x, k)=-\sum\limits_{n=0}^{\infty}\frac{(2n)!}{2^{2n}n!^2(2n-1)}\,k^{2n}I_{2n}(x),
\end{align}
where
$$
I_{2n}(x)=\int\limits_0^x\frac{t^{2n}}{\sqrt{1-t^2}}\,dt\,.
$$
We see that
$$
I_0(x)=\arcsin x, \quad I_2(x)=\frac{1}{2}\left(\arcsin x-x\sqrt{1-x^2}\right),
$$
$$
I_{2n}(x)=\frac{2n-1}{2n}I_{2(n-1)}(x)-\frac{1}{2n}x^{2n-1}\sqrt{1-x^2}, \quad n\ge 1.
$$
This implies that
\begin{equation}\label{I02}
I_0(x)-2I_2(x)=x\sqrt{1-x^2}>0, \quad 0<x<1.
\end{equation}
Moreover,
\begin{equation}\label{Ipi}
 0<I_{2n}(x)\le I_{0}(1)=\frac{\pi}{2}, \quad n\ge 0, \quad 0<x<1.
\end{equation}

From \eqref{k1} it follows that
$$
\frac{k^2}{\lambda^2}=\frac{E'(k)}{K'(k)}.
$$
Thus, for $k=k(H)$, $\lambda=\lambda(H)$ and $l=l(H)$ we have
\begin{align*}
E(l',k')-(k^2/\lambda^2)F(l',k')&=\frac{K(k')E(l',k')-E(k')F(l',k')}{K(k')}\\&\sim\frac{2}{\pi}\left[ K(k')E(l',k')-E(k')F(l',k')\right],\quad\text{ as }\quad H\to \infty.
\end{align*}
By applying  \eqref{exp_int1} and \eqref{exp_int2} with parameters $l'$ and $k'$, instead of $x$ and $k$, we get
\begin{align*}
E(l', k')&=I_0(l')-\frac{1}{2}I_2(l')k'^2-\frac{1}{8}I_4(l')k'^4-\ldots\,,\\
F(l', k')&=I_0(l')+\frac{1}{2}I_2(l')k'^2+\frac{3}{8}I_4(l')k'^4+\ldots\, .
\end{align*}
Taking into account \eqref{Ipi}, we conclude that
\begin{align*}
E(l', k')&=I_0(l')-\frac{1}{2}I_2(l')k'^2+o(k'^2),\\
F(l', k')&=I_0(l')+\frac{1}{2}I_2(l')k'^2+o(k'^2),
\end{align*}
uniformly with respect to $l'$, as $k'\to 0$.

This together with \eqref{exp_K} and \eqref{exp_E} implies
\begin{align}\label{denoH}
K(k')E(l',k')-E(k')F(l',k') =\frac{\pi}{4}\left[I_0(l')-2I_2(l')\right]k'^2+ o(k'^2)
\end{align}
uniformly with respect to $l'$, as $k'\to 0$.  Since $l'=l'(H)\to \sqrt{\delta_2}\in(0,1)$, as $H\to\infty$, we have
\begin{align}\label{denoH}
K(k')E(l',k')-E(k')F(l',k') \sim\frac{\pi}{4}\left[I_0(\sqrt{\delta_2})-2I_2(\sqrt{\delta_2})\right]k'^2,
\end{align}
for $k'=k'(H)$, $l'=l'(H)$, as $H\to \infty$. We note that $I_0(\sqrt{\delta_2})-2I_2(\sqrt{\delta_2})\neq 0$ in \eqref{denoH} because of \eqref{I02}.

From \eqref{kkprime}, \eqref{nenoH} and \eqref{denoH} we obtain for $k'=k'(H)$
\begin{align*}
\frac{H\alpha}{\beta}\sim\frac{2C}{k'^2}, \quad\text{ as }\quad H\to\infty,
\end{align*}
with some constant $C\neq 0$.
This implies that
\begin{align}\label{asym_H}
\frac{1}{\pi}\log H\sim\frac{1}{\pi}\log\frac{1}{1-k},
\quad\text{ as \quad $k\to 1$}.
\end{align}

Consider the quadrilateral $\mbox{\boldmath$G$}^+_{2H}=(G^+_{2H}; H\alpha +i\beta, (H\alpha )^+, (-H\alpha)^-, -H\alpha +i\beta)$ where $G^+_{2H}$ is the upper half of $G_{2H}$, $(-H\alpha )^-$ and $(H\alpha )^+$ are  the prime ends of $G^+_{2H}$  with supports at the points
$-H\alpha $ and $H\alpha $, lying in the left and right edges of  the slits along the segments $[-H\alpha,-H\alpha+i\beta]$ and $[H\alpha,H\alpha+i\beta]$,
 respectively. Since it is conformally equivalent to the quadrilateral which is the upper half-plane with vertices $-1/k$, $-1/\lambda$, $1/\lambda$ and $1/k$,
 we conclude  that
$$
\operatorname{Mod}(\mbox{\boldmath$G$}^+_{2H})= \frac{2K(k/\lambda)}{K'(k/\lambda)}.
$$

With the help of the symmetry principle, we see that
$$\operatorname{Mod}(\mbox{\boldmath$G$}_{2H})=\frac{K(k/\lambda)}{K'(k/\lambda)}\sim \frac{1}{\pi}\log\frac{1}{1-k/\lambda}, \quad k=k(H),\quad \lambda=\lambda(H),\quad\text{ as }\quad  H\to\infty.
$$
At last, $1-k(H)/\lambda(H) \sim \lambda(H)-k(H)=(1-k(H))-(1-\lambda(H))\sim (1-\delta_2)(1-k(H))$, as $H\to\infty$, therefore,
$$
\frac{1}{\pi}\log\frac{1}{1-k/\lambda}\sim \frac{1}{\pi}\log\frac{1}{1-k}
$$
and this, together with \eqref{asym_H} implies \eqref{mod_asym_G2+H}.
\end{proof}

\textit{Proof of Theorem~\ref{mt}.}  From \eqref{compar}, Lemmas~\ref{lem_asym_G1H} and \ref{lem_asym_G2+H} it follows \eqref{asqh}.

\section{Discussion}
There are still many interesting issues concerning investigation of
the behavior of conformal moduli of domains under their geometric
transformations. Our next task includes investigating the validity of Conjecture\ref{conject_nasyrov} suggested by Prof. S.~R.~Nasyrov for the case of an sufficiently arbitrary nonsymmetric
quadrilateral.

\section*{FUNDING}
The work of the first author was performed under the development program of
the Volga Region Mathematical Center (agreement no. 075-02-2022-882).

%%%%%%%%%%%%%%%%%%

\end{document}